\documentclass[a4paper,12pt]{article}
\usepackage{amssymb,amsmath,amsfonts,amsthm}
\usepackage[all]{xy}
\usepackage{enumerate}
\usepackage{mathrsfs}
\usepackage{graphicx}
\usepackage{authblk}
\usepackage{blindtext}
\usepackage[english]{babel}
\usepackage[numbers,comma,sort&compress]{natbib}
\usepackage{url}
\makeatletter
\renewcommand\@biblabel[1]{#1.} 
\makeatother
\usepackage[a4paper]{geometry}
\newcommand{\Z}{\mathbb{Z}}

\newcommand{\PP}{\mathbb{P}}
\newcommand{\Bc}{\mathcal{B}}

\newcommand{\point}{\mathfrak{p}}

\newcommand{\RI}{\mathcal{R_I}}
\newcommand{\SI}{\mathcal{S_I}}
\newcommand{\m}{\mathfrak{m}}
\newcommand{\Y}{\mathcal{Y}}
\newcommand{\K}{\mathcal{K}}
\newcommand{\ZC}{\mathcal{Z}}

\newcommand{\Ker}{{\operatorname{Ker}}}

\newcommand{\Supp}{{\operatorname{Supp}}}

\newcommand{\sat}{{\operatorname{sat}}}
\newcommand{\reg}{{\operatorname{reg}}}
\newcommand{\indeg}{{\operatorname{indeg}}}

\newcommand{\Hom}{{\operatorname{Homgr}}}
\newcommand{\codim}{{\operatorname{codim}}}
\newcommand{\Proj}{{\operatorname{Proj}}}

\newcommand{\Sym}{{\operatorname{Sym}}}
\newcommand{\Rees}{{\operatorname{Rees}}}
\newcommand{\length}{{\operatorname{length}}}

\newcommand{\pmt}[1]{\begin{pmatrix}#1\end{pmatrix}}

\def\ff{{\bf f}} 

\numberwithin{equation}{section} 
\usepackage{hyperref}

\usepackage{xpatch}
\makeatletter
\AtBeginDocument{\xpatchcmd{\@thm}{\thm@headpunct{.}}{\thm@headpunct{}}{}{}}
\makeatother
\newtheorem{pro}{Proposition}[section]
\newtheorem{Lem}[pro]{Lemma}
\newtheorem{Cor}[pro]{Corollary} 
\newtheorem{Theo}[pro]{Theorem}

\newtheorem*{Theoetoile}{Theorem} 

\newtheorem*{propositionetoile}{Proposition} 
\theoremstyle{definition}

\newtheorem{Exem}[pro]{Example}

\newtheorem{Rem}[pro]{Remark} 
 
\def\ff{{\mathbf{f}}}

\usepackage{hyperref}
\hypersetup{
	colorlinks=false,       
	linkcolor=red,          
	citecolor=blue,        
	filecolor=magenta,      
	urlcolor=cyan           
}

\begin{document}
\date{}

\title{\Large \bf Fibers of rational maps and  Rees algebras of their base ideals}
\author[1]{Tran Quang Hoa \thanks{Corresponding Author: tranquanghoa@hueuni.edu.vn}}
\author[2]{ Ho Vu Ngoc Phuong \thanks{hvnphuong@husc.edu.vn}}
\affil[1]{\small University of Education, Hue University,  34 Le Loi St., Hue, Vietnam.}
\affil[2]{\small University of Sciences, Hue University,  77 Nguyen Hue St., Hue, Vietnam.}

\maketitle

\begin{abstract}
We consider a rational map $\phi: \PP_k^{m} \dashrightarrow \PP_k^n$ that is a parameterization of an $m$-dimensional variety. Our main goal  is to study  the $(m-1)$-dimensional fibers of $\phi$ in relation to the $m$-th local cohomology modules of the Rees algebra of its base ideal.
\end{abstract}

\section{Introduction}
Let $k$ be a field and $\phi: \PP_k^m\dashrightarrow \PP_k^n$ be a rational map. Such a map $\phi$ is defined by homogeneous polynomials $f_0,\ldots,f_n,$ of the same degree $d,$ in a standard graded polynomial  ring $R=k[X_0,\ldots,X_m],$ such that $\gcd(f_0,\ldots,f_n)=1.$ The ideal $I$ of $ R$ generated by these polynomials is called the \textit{base ideal of $\phi$}.  The scheme $\Bc:=\Proj(R/I)\subset \PP_k^m $ is  called the \textit{base locus of $\phi$}. Let $B=k[T_0,\ldots,T_n]$ be the homogeneous coordinate ring of $\PP_k^n.$ The map $\phi$ corresponds to the $k$-algebra homomorphism $\varphi: B\longrightarrow R,$ which sends each $T_i$ to $f_i.$ Then, the kernel of this homomorphism defines the closed image $\mathscr{S}$  of $\phi.$ In other words,  after degree renormalization, $k[f_0,\ldots,f_n]\simeq B/\Ker(\varphi)$ is the homogeneous coordinate ring of $\mathscr{S}.$  The minimal set of generators of $\Ker(\varphi)$ is called its \textit{implicit equations} and the \textit{implicitization problem} is to find these implicit equations.

The implicitization problem  has been of increasing interest to commutative algebraists and algebraic geometers due to its applications in Computer Aided Geometric Design as explained by Cox \cite{Cox05}. 

We blow up the base locus of $\phi$ and obtain the following commutative diagram
\begin{displaymath}
\xymatrix{ \Gamma \ar@{^{(}->}[rr]\ar[d]_{\pi_1}&& \PP_k^m\times \PP_k^n \ar[d]^{\pi_2} \\ \PP_k^m \ar@{-->}[rr]^{\phi} &&\PP_k^n.     }
\end{displaymath}
The variety $\Gamma$ is the blow-up of $\PP_k^m$ along $\Bc$, and it is also the Zariski closure of the graph of $\phi$ in $\PP^m_k\times \PP_k^n.$ Moreover, $\Gamma$ is the geometric version of the Rees algebra $\RI$ of $I,$ i.e., $\Proj(\RI)=\Gamma.$ As $\RI$ is the graded domain defining $\Gamma,$ the projection  $\pi_2(\Gamma)=\mathscr{S}$  is defined by the graded domain $\RI\cap k[T_0,\ldots,T_n]$, and we can thus obtain the implicit equations of $\mathscr{S}$  from the defining equations of $\RI.$

Besides the computation of implicit representations of parameterizations, in geometric modeling  it is of vital importance to have a detailed knowledge of the geometry of the object and of the parametric representation one is working with.  The question of how many times is the same point being painted (i.e., corresponds to distinct values of parameter) depends  not only on the variety itself but also on the parameterization.  It is of  interest for applications to determine the singularities of the parameterizations. The main goal of this paper is to study  the fibers of parameterizations in relation to the Rees and symmetric algebras of their base ideals.  More precisely, we set
$$\pi:={\pi_2}_{\mid \Gamma}: \Gamma \longrightarrow \PP_k^n.$$ 
For every closed point $y\in \PP_k^n,$ we will denote  its residue field by $k(y)$. If $k$ is assumed to be  algebraically closed, then $k(y)\simeq k.$ The fiber of $\pi$  at  $y\in \PP_k^n$ is the subscheme
\begin{align*}
\pi^{-1}(y)=\Proj(\RI\otimes_B k(y)) \subset \PP_{k(y)}^m\simeq \PP_k^m.
\end{align*}
Suppose that  $m\geq 2$ and $\phi$ is generically finite onto its image.  Then, the set
$$\Y_{m-1}=\{y\in \PP_k^n \mid \dim  \pi^{-1}(y)=m-1\}$$ 
consists of only  a finite number of points in $\PP_k^n.$  For each $y\in \Y_{m-1}$, the fiber of $\pi$ at $y$ is an $(m-1)$-dimensional subscheme of $\PP_k^m$ and thus the unmixed component of maximal dimension is defined by a homogeneous polynomial $h_y\in R.$ One of the interesting problems is to establish an upper bound for $\sum_{y\in \Y_{m-1}}\deg(h_y)$ in terms of $d$. This problem was studied in \cite{CCT17, QHTran17}.

The paper is  organized as follows. In Section 2, we study the structure of $\Y_{m-1}.$ Some results in this section were proved in \cite{CCT17}. The main result of this section is Theorem~\ref{Theorem1.3} that gives an upper bound for $\sum_{y\in \Y_{m-1}}\deg(h_y)$ by the initial degree of certain symbolic powers of its base ideal. This is a  generalization of \cite[Proposition 1]{QHTran17} where the first author only proved this result for parameterizations of surfaces $\phi: \PP_k^2 \dashrightarrow \PP_k^3$ under the assumption that the base locus $\Bc$ is locally a complete intersection. More precisely, we have the following.
\begin{Theoetoile}
If there exists an integer $s$ such that $\nu=\indeg((I^s)^\sat)<sd$, then
$$\sum_{y\in \Y_{m-1}}\deg(h_y)\leq \nu<sd.$$
In particular, if $\indeg (I^\sat)< d$, then $\sum_{y\in \Y_{m-1}}\deg(h_y)<d.$
\end{Theoetoile}

In Section 3, we study the part of graded $m$ in $X_i$ of the $m$-th local cohomology modules of the Rees algebra with respect to the homogeneous maximal ideal $\m=(X_0,\ldots,X_m)$
$$N=H_\m^m(\RI)_{(-m,\ast)}= \oplus_{s\geq 0}H_\m^m(I^s)_{sd-m}.$$
The main result  of this section is the following. 
\begin{Theoetoile} [Theorem~\ref{Theorem3.2}]
We have that $N$ is a finitely generated $B$-module satisfying
\begin{enumerate}
\item [\rm (i)] $\Supp_B(N)=\Y_{m-1}$ and $\dim (N)=1.$
\item [\rm (ii)] $\deg (N)=\sum_{y\in \Y_{m-1}}\binom{\deg(h_y)+m-1}{m}.$
\end{enumerate}
\end{Theoetoile}
In the last section, we treat the case of parameterization $\phi: \PP_k^2 \dashrightarrow \PP_k^3$ of surfaces. We establish a bound for the Castelnuovo-Mumford regularity and the degree of the $B$-module 
$$N=\oplus_{s\geq 0}H_\m^2(I^s)_{sd-2},$$
see Corollary~\ref{Cor4.3} and Proposition~\ref{Lemma4.2}.
\begin{propositionetoile}
Assume $\Bc=\Proj(R/I)$ is locally a complete intersection.  Then 
$$\reg(N)\leq n\quad \text{and}\quad \deg(N)\leq \binom{n+2}{3},$$
where $n=\dim_k H^1_\m (R/I)_{d-2}$. Moreover, if $\indeg(I^\sat)=d$, then 
$$d\leq n\leq \dfrac{d(d-3)}{2}+3.$$
\end{propositionetoile}

\section{Fibers of rational maps $\phi: \PP_k^m \dashrightarrow \PP_k^n$ }
Let $n\geq m\geq 2$ be integers and $R=k[X_0,\ldots,X_m]$ be the standard graded polynomial ring over an algebraically closed field $k$. Denote  the homogeneous maximal ideal of $R$ by $\m=(X_0,\ldots,X_m)$. Suppose we are given an integer $d\geq 1$ and $n+1$ homogeneous polynomials $f_0,\ldots,f_n\in R_d,$ not all zero. We may further assume that $\gcd(f_0,\ldots,f_n)=1,$ replacing the $f_i's$ by their quotient by the  greatest common divisor of $f_0,\ldots,f_n$ if needed; hence, the ideal $I$ of $R$ generated by these polynomials is of codimension at least two. Set $\Bc:=\Proj(R/I)\subseteq \PP_k^m:=\Proj(R)$ and consider the rational map
\begin{align*}
\phi:\quad &\PP_k^m  -\dashrightarrow \PP_k^n\\
& \; x \longmapsto (f_0(x):\cdots: f_n(x))
\end{align*}
whose closed image is the subvariety $\mathscr{S}$ in $\PP_k^n.$ In this paper, we always assume that $\phi$ is generically finite onto its image, or equivalently that the closed image $\mathscr{S}$ is of dimension $m.$ In this case, we say that $\phi$ is a \textit{parameterization} of the $m$-dimensional variety $\mathscr{S}$.

Let $\Gamma_0\subset \PP_k^m\times \PP_k^n$ be the graph of $\phi: \PP_k^m\setminus\Bc \longrightarrow \PP_k^n$ and $\Gamma$ be the Zariski closure of $\Gamma_0$. We have the following diagram
$$ \xymatrix@1{\Gamma\ \ar[d]_{\pi_1}\ar@{^(->}[r] & \PP_k^m\times \PP_k^n \ar[d]^{\pi_2}\\  \PP^m_k \ar@{-->}[r]_\phi& \PP^n_k}	$$
where the maps $\pi_1$ and $\pi_2$ are the canonical projections. One has
$$\mathscr{S}=\overline{\pi_2(\Gamma_0)}=\pi_2(\Gamma),$$
where the bar denotes the Zariski closure. Furthermore, $\Gamma$ is the irreducible subscheme of $\PP_k^m\times \PP_k^n$ defined by the Rees algebra 
$$\RI:=\Rees_R(I)=\oplus_{s\geq 0} I^s.$$
Denote the homogeneous coordinate ring of $\PP_k^n$ by $B:=k[T_0,\ldots,T_n]$. Set 
$$S:=R\otimes_k B= R[T_0,\ldots,T_n]$$
with the grading $\deg(X_i)=(1,0)$ and $\deg(T_j)=(0,1)$ for all $i=0,\ldots,m$ and $j=0,\ldots,n$. The natural bi-graded morphism of $k$-algebras
\begin{align*}
\alpha:\quad & S \longrightarrow  \RI=\oplus_{s\geq 0} I(d)^s=\oplus_{s\geq 0} I^s(sd)\\
&  T_i \longmapsto f_i
\end{align*}
is onto and corresponds to the embedding $\Gamma \subset \PP_k^m\times \PP_k^n$.

Let $\mathfrak{P}$ be the kernel of $\alpha$. Then, it is a bi-homogeneous ideal of $S$, and the part of degree one of  $\mathfrak{P}$ in $T_i$, denoted by  $\mathfrak{P}_1=\mathfrak{P}_{(\ast, 1)},$  is the module of syzygies of the  $I$
$$a_0T_0 +\cdots+a_nT_n \in \mathfrak{P}_1\Longleftrightarrow a_0f_0+\cdots+a_nf_n=0.$$
Set $\SI:=\Sym_R(I)$ for the symmetric algebra of $I$. The  natural bi-graded epimorphisms
\begin{align*}
S\longrightarrow S/(\mathfrak{P}_1)\simeq \SI \text{\quad and\quad} \delta:\SI\simeq S/(\mathfrak{P}_1)\longrightarrow S/\mathfrak{P}\simeq \RI
\end{align*}
correspond to the embeddings of schemes $ \Gamma \subset V\subset \PP_k^m\times \PP_k^n,$ where $V$ is the projective scheme defined by $\SI$.

Let $\K$ be the kernel of $\delta,$ one has the following exact sequence 
$$0\longrightarrow \K\longrightarrow \SI\longrightarrow \RI\longrightarrow 0.$$
Notice that the module $\K$ is supported in $\Bc$ because $I$ is locally trivial outside $\Bc$.

As the construction of symmetric and Rees algebras commutes with
localization, and both algebras are the quotient of a polynomial extension of the
base ring by the Koszul syzygies on a minimal set of generators in the case of a
complete intersection ideal, it follows that $\Gamma$ and $V$ coincide on $(\PP_k^m\setminus X) \times \PP_k^n,$ where $X$ is the (possibly empty) set of points where $\Bc$ is not locally a complete intersection.

Now we set $\pi:={\pi_2}_{\mid \Gamma}: \Gamma \longrightarrow \PP_k^n.$ For every closed point $y\in \PP_k^n,$ we will denote its residue field by $k(y)$, that is, $k(y)=B_\point/\point B_\point,$ where $\point$ is the defining prime ideal of $y.$ As $k$ is algebraically closed, $k(y)\simeq k.$ The fiber of $\pi$  at  $y\in \PP_k^n$ is the subscheme
\begin{align*}
\pi^{-1}(y)=\Proj(\RI\otimes_B k(y)) \subset \PP_{k(y)}^m\simeq \PP_k^m.
\end{align*}
Let $0\leq \ell\leq m,$ we define 
$$\Y_\ell=\{y\in \PP_k^n \mid \dim  \pi^{-1}(y)=\ell\}\subset \PP_k^n.$$
Our goal is to study the structure of $\Y_\ell.$ Firstly,  we have the following.
\begin{Lem}{\rm \cite[Lemma 3.1]{CCT17}} \label{Lemma2.2.1}
Let $\phi: \PP_k^m\dashrightarrow \PP_k^n$ be a parameterization of $m$-dimensional variety and $\Gamma$ be the closure of the graph of $\phi.$ Consider  the projection $\pi: \Gamma\longrightarrow \PP_k^n$. Then
$$\dim \ \overline{\Y_\ell}+\ell \leq m.$$
Furthermore, this inequality is strict for any $\ell > 0.$ As a consequence, $\pi$ has no $m$-dimensional fibers and only has a finite number of $(m-1)$-dimensional fibers.
\end{Lem}

The fibers of $\pi$ are defined by the specialization of the Rees algebra. However, Rees algebras are difficult to study. Fortunately, the symmetric algebra of $I$ is easier to understand than $\RI$, and the fibers of $\pi$ are closely related to the fibers of $$\pi^\prime:={\pi_2}_{\mid V}: V \longrightarrow \PP_k^n.$$
Recall that  for any  closed point $y\in \PP_k^n,$ the fiber of $\pi^\prime$ at $y$ is the subscheme
\begin{align*}
{\pi^\prime}^{-1}(y)=\Proj(\SI\otimes_B k(y))\subset \PP_{k(y)}^m\simeq \PP_k^m.
\end{align*}  

The next result gives a relation between fibers of $\pi$ and $\pi^\prime$ -- recall that $X$ is the (possible empty) set of points where $\Bc$ is not locally a complete intersection. 

\begin{Lem}{\rm \cite[Lemma 3.2]{CCT17}} \label{Lemma2.2.2}
The fibers of $\pi$ and $\pi^\prime$ agree outside $X,$ hence  they have the same $(m-1)$-dimensional fibers.
\end{Lem}

The next result is a generalization of \cite[Lemma 10]{BBC14} that  gives the structure of the unmixed part of a $(m-1)$-dimensional fiber of $\pi.$ Note that our result does not need the assumption that $\Bc$ is locally a complete intersection as in \cite{BBC14}, thanks to Lemma~\ref{Lemma2.2.2}. Recall that the saturation of an ideal $J$ of $R$ is defined  by $J^\sat:=J \colon_R (\m)^\infty$. 
\begin{Lem} \cite[Lemma 3.3]{CCT17}\label{Lemma1.2}
Assume $y= (p_0: \cdots: p_n)\in \Y_{m-1}$ such that $p_i=1$. Then the unmixed part of the fiber $\pi^{-1}(y)$ is defined by 
$$h_y= \gcd(f_0-p_0f_i,\ldots,f_n-p_nf_i).$$
Furthermore, if  $f_j-p_jf_i=h_y g_j$ for all $j\neq i$, then $$ I=(f_i)+h_y(g_0,\ldots,g_{i-1},g_{i+1},\ldots,g_n)\quad \text{and}\quad I^\sat\subset (f_i,h_y).$$ 
\end{Lem}

\begin{Rem}
The above lemma shows that the $(m-1)$-dimensional fibers of $\pi$ can only occur when $\Bc \neq \emptyset$ as $\Bc \supset V(f_i,h_y).$ It also shows that
$$d\deg(h_y)\leq \deg(\Bc),$$
if there is a $(m-1)$-dimensional fiber with unmixed part given by $h_y.$ As a consequence, $\deg(h_y) <d$ for any $y\in \Y_{m-1}.$
\end{Rem}

 By Lemma~\ref{Lemma2.2.1}, $\pi$ only has a finite number of $(m-1)$-dimensional fibers. The following give an upper bound for this number in terms of the initial degree of certain symbolic powers of its base ideal.  Recall that the initial degree of a graded $R$-module $M$ is defined by
 $$\indeg(M):=\inf\{n\in \Z \mid M_n\neq 0\}$$
 with convention that $\sup \emptyset =+\infty.$ 
\begin{Theo}\label{Theorem1.3}
If there exists an integer $s\geq 1$ such that $\nu=\indeg((I^s)^\sat)<sd$, then
$$\sum_{y\in \Y_{m-1}}\deg(h_y)\leq \nu<sd.$$
In particular, if $\indeg (I^\sat)< d$, then
$$\sum_{y\in \Y_{m-1}}\deg(h_y)<d.$$
\end{Theo}
\begin{proof}
As $\Y_{m-1}$ is finite, by Lemma~\ref{Lemma1.2}, there exists a homogeneous polynomial $f\in I$ of degree $d$ such that, for any $y\in \Y_{m-1},$
$$I=(f)+h_y (g_{1y}, \ldots,g_{ny}) \quad \text{and}\quad I^\sat \subset (f,h_y)$$
for some $ g_{1y}, \ldots,g_{ny}\in R$.  Since $(f,h_y)$ is a complete intersection ideal, it follows from \cite[Appendix~6, Lemma~5]{Zariski1960} that $(f,h_y)^s$ is unmixed, hence saturated for every integer $s\geq 1$. Therefore, for all $y\in \Y_{m-1}$, 
$$(I^s)^\sat\subset ((I^\sat)^s)^\sat \subset ((f,h_y)^s)^\sat=(f,h_y)^s=(f^s,f^{s-1}h_y,\ldots, h_y^s).$$

Now, let $0\neq F\in (I^s)^\sat$  such that $\deg(F)=\nu<sd$, then $h_y$ is a divisor of $F$. Moreover, if $y\neq y^\prime$ in $\Y_{m-1}$, then $\gcd(h_y,h_{y^\prime})=1.$ We deduce that
$$\prod_{y\in \Y_{m-1}}h_y \mid F$$
which gives
$$\sum_{y\in \Y_{m-1}}\deg(h_y)\leq \deg(F)=\nu<sd.$$
\end{proof}
\begin{Rem}
In the case where  $\phi: \PP_k^2 \dashrightarrow \PP_k^3$ is a parameterization of surfaces. In \cite{QHTran17}, the first author showed  that if $\Bc$ is locally a complete intersection of dimension zero, then 
\begin{displaymath}
\sum_{y\in \Y_1}\deg(h_y)\leq \begin{cases} 
4  & \text{if}  \quad d=3,\\
\left\lfloor \frac{d}{2} \right \rfloor  d-1 &\text{if} \quad d\geq 4.
\end{cases}
\end{displaymath} 
\end{Rem}
\begin{Exem}
Consider the parameterization $\phi: \PP_k^2 \dashrightarrow \PP_k^3$ of surface  given by
\begin{equation*}
\begin{array}{cc}
f_0&= X_0X_1(X_0-X_2)(X_0+X_2)(X_0-2X_2)\\
f_1&=X_0X_1(X_1-X_2)(X_1+X_2)(X_1-2X_2)\\
f_2&=X_0X_2(X_0-X_2)(X_0+X_2)(X_0-2X_2)\\
f_3&=X_1X_2(X_1-X_2)(X_1+X_2)(X_1-2X_2).
\end{array}
\end{equation*}
Using \texttt{Macaulay2} \cite{Macaulay2}, it is easy to see that $I=I^\sat$ and $\indeg((I^2)^\sat) =8<2.5=10.$ Furthermore, $I$ admits a free resolution
\begin{displaymath}
\xymatrix{0\ar[r]& R(-6)^2\oplus R(-8)\ar[r]^{\qquad M}& R(-5)^4\ar[r]& R\ar[r]& R/I\ar[r]&0,}
\end{displaymath} 
where matrix $M$ is given by
$$\pmt{-X_2&0&(X_1-X_2)(X_1+X_2)(X_1-2X_2)\\0&-X_2&-(X_0-X_2)(X_0+X_2)(X_0-2X_2)\\X_1&0&0\\0&X_0&0} .$$
Thus, we obtain $\Y_1=\{\point_1,\point_2,\point_3,\point_4,\point_5,\point_6,\point_7,\point_8\}$	with
\begin{align*}
\begin{array}{llll}
\point_1 =(0:0:0:1)& h_{\point_1}=X_0& \point_2=(0:0:1:0)& h_{\point_2}=X_1\\
\point_3  =(0:1:0:1)& h_{\point_3}  =X_0-X_2& \point_4=(0:-1:0:1)& h_{\point_4}=X_0+X_2\\
\point_5  =(0:2:0:1)& h_{\point_5}  =X_0-2X_2& \point_6=(1:0:1:0)& h_{\point_6}=X_1-X_2\\
\point_7  =(-1:0:1:0)& h_{\point_7}  =X_1+X_2& \point_8=(2:0:1:0)& h_{\point_8}=X_1-2X_2.
\end{array}
\end{align*}
Consequently, we have
$$\sum_{y\in \Y_1}\deg(h_y)=8=\indeg((I^2)^\sat).$$
\end{Exem}
 
\section{Local cohomology of Rees algebras of the base ideal of parameterizations} 
Let $\phi: \PP_k^m\dashrightarrow \PP_k^n$ be a parameterization of $m$-dimensional variety. Let $R=k[X_0,\ldots,X_m]$ and  $B=k[T_0,\ldots,T_n]$ be the homogeneous coordinate ring of $\PP_k^m$ and $\PP_k^n$, respectively. For every closed point $y\in \PP_k^n$, the fiber of $\pi$  at  $y$ is the subscheme
\begin{align*}
\pi^{-1}(y)=\Proj(\RI\otimes_B k(y)) \subset \PP_{k(y)}^m\simeq \PP_k^m
\end{align*}
and  we are interested in studying the set
$$\Y_{m-1}=\{y\in \PP_k^n\, \mid\, \dim \pi^{-1}(y)=m-1\}.$$
We now consider  the $B$-module
 $$M_\mu=H_\m^m(\RI)_{(\mu,\ast)}= \oplus_{s\geq 0}H_\m^m(I^s)_{\mu+sd},$$
 where $\m=(X_0,\ldots,X_m)$ is the homogeneous maximal ideal of $R$.
 By \cite[Theorem 2.1]{Chardin2013}, $M_\mu$ is a finitely generated $B$-module for all $\mu\in \Z$. The following result gives a relation between the support of $M_\mu$ and $\Y_{m-1}.$ For each $y\in \PP_k^n=\Proj(B)$, we can see $y$ as a homogeneous prime ideal of $B.$
\begin{pro} \label{Proposition3.1}
One has
$$ \Supp_B(M_\mu)=\{y\in \Y_{m-1}\mid   \deg(\pi^{-1}(y))\geq \mu+m+1\}.$$
\end{pro}
\begin{proof}
As $k$ is algebraically closed, we have 
\begin{align*}
\pi^{-1}(y)=\Proj(\RI\otimes_B k(y)) \subset \PP_{k(y)}^m\simeq \PP_k^m.
\end{align*}
Therefore, the homogeneous coordinate ring of  $\pi_2^{-1}(y)$ is
$$\RI\otimes_Bk(y)\simeq R/J,$$
where $J$ is a satured ideal of $R$ depending on  $y.$  Let $y\in \Y_{m-1}$. As $\dim\pi^{-1}(y)=m-1$, one has
$$\dim (\RI\otimes_Bk(y))=\dim R/J =m.$$
Since $\dim R=m+1$, there  exists a homogeneous polynomial $f$ of degree $d_f$ such that $J=(f)J^\prime$, with $\codim (J')\geq 2.$ Notice that $f$ is exactly the defining equation of unmixed part of $\pi^{-1}(y).$ 
Consider the exact sequence
$$0\longrightarrow (f)/J\longrightarrow R/J\longrightarrow R/(f)\longrightarrow 0$$
which deduces the exact sequence in cohomology
$$0=H_\m^m((f)/J)\longrightarrow H_\m^m(R/J)\longrightarrow H_\m^m(R/(f))\longrightarrow H_\m^{m+1}((f)/J)=0,$$
since  $\codim (J')\geq 2$, hence $(f)/J$ is of dimension at most $m-1.$ It follows from the above exact sequence that
\begin{align}\label{equation3.1}
H_\m^m(\RI\otimes_Bk(y))\simeq H_\m^m(R/J)\simeq H_\m^m(R/(f)).
\end{align}
We consider the following exact sequence
\begin{align*}
\xymatrix{0\ar[r] & R[-d_f] \ar[r]^{\quad \times f} & R \ar[r] & R/(f) \ar[r] & 0}
\end{align*}
which implies the exact sequence in cohomology
\begin{align*}
\xymatrix{0=H_\m^m(R)\ar[r] & H_\m^m(R/(f)) \ar[r] & H_\m^{m+1}(R[-d_f]) \ar[r]^{\quad\times f} & H_\m^{m+1}(R) \ar[r] & 0.}
\end{align*}
In degree $\mu$, one has the following exact sequence
\begin{equation}\label{suite5}
0\longrightarrow H_\m^m(R/(f))_\mu\longrightarrow H_\m^{m+1}(R)_{\mu-d_f}\longrightarrow H_\m^{m+1}(R)_\mu\longrightarrow 0.
\end{equation}
On the other hand,
$$H_\m^{m+1}(R)\simeq (X_0\cdots X_m)^{-1}k[X_0^{-1},\ldots,X_m^{-1}],$$
hence 
$$H_\m^{m+1}(R)_\mu\simeq (R^\ast)_{\mu+m+1}:=\Hom_R(R_{-\mu-m-1},k).$$
It follows that
$H_\m^{m+1}(R)_\mu=0$ for all $\mu>-m-1$ and $H_\m^{m+1}(R)_\mu\neq 0$ for any $\mu \leq -m-1$. It follows from \eqref{suite5} that
\begin{align}\label{equation3.3}
H_\m^m(R/(f))_\mu =\begin{cases}
0&\text{if} \quad \mu>d_f-m-1\\
(R/(f))^\ast_{\mu-d_f+m+1}\neq 0&\text{if}\quad  \mu\leq d_f-m-1.
\end{cases} 
\end{align}

By definition, $M_\mu$ is a graded $B$-module and $\Supp_B (M_\mu)\subset \Proj(B).$ Now let $\point\in \Proj(B)$, we have
\begin{align*}
\point\in \Supp_B(M_\mu)&\iff  M_\mu \otimes_B B_\point\neq 0\\
& \iff M_\mu\otimes_B B_\point\otimes_B (B/\point)\neq 0 \\
&\iff H_\m^m(\RI)_{(\mu,\ast)}\otimes_Bk(\point)\neq 0\\
&\iff H_\m^m(\RI\otimes_Bk(\point))_{(\mu,\ast)}\neq 0.
\end{align*}
In particular, $H_\m^m(\RI\otimes_Bk(\point))\neq 0$, hence $\dim (\RI\otimes_Bk(\point)) =m$ which shows that $\point\in \Y_{m-1}.$ It follows from \eqref{equation3.1} and \eqref{equation3.3} that $\deg(\pi^{-1}(\point))\geq \mu+m+1.$
\end{proof}
 
In particular, if $\mu=-m$,  then the finitely generated $B$-module
$$N=\oplus_{s\geq 0}H_\m^m(I^s)_{sd-m}$$
satisfies $\Supp_B(N) =\Y_{m-1}$ by Proposition~\ref{Proposition3.1}.
Furthermore, we have the following.
\begin{Theo} \label{Theorem3.2}
Let $N$ be the finitely generated $B$-module as above. Then
\begin{enumerate}
\item [\rm (i)] $\dim (N)=1.$
\item [\rm (ii)] $\deg (N)=\sum_{y\in \Y_{m-1}}\binom{\deg(h_y)+m-1}{m}.$
\end{enumerate}
\end{Theo}
\begin{proof}
Let $y =(p_0:p_1:\cdots:p_n)\in \Y_{m-1}.$ Without loss of generality, we can assume that $p_0=1.$ Hence,
$$\point=(T_1-p_1T_0,\ldots,T_n-p_nT_0)\subset B$$
is the defining ideal of $y$. For any $f\in B$, we have
$$f=g_1(T_1-p_1T_0)+\cdots+g_n(T_n-p_nT_0)+ v\; \text{for some}\; v\in k[T_0].$$
It follows that $f+\point =v+\point.$ This implies that $B/\point\simeq k[T_0].$ Therefore,
$$\dim(B/\point)=1\quad \text{for any}\; \point\in \Y_{m-1}$$
and thus,
$$\dim(N)=\max_{\point\in \Supp_B(N)}\dim (B/\point)=1$$
which shows (i). We now prove for item (ii). It was known that
$$HP_N(s)=HF_N(s)=\dim_kN_s=\dim_kH_\m^m(I^s)_{sd-m}$$
for all $s\gg 0$, where $HP_N$ and $HF_N$ is the Hilbert polynomial and the Hilbert function of $N$, respectively. As  $\dim N=1$, the Hilbert polynomial of $N$ is constant which is equal to $\deg(N).$  On the other hand,
$$\deg(N)=\sum_{\dim(B/\point)=1} \length_{B_\point}(N_\point).\deg(B/\point).$$
We proved that $B/\point\simeq k[T_0],$ hence $\dim(B/\point)=1$ and $\deg(B/\point)=1$ for the defining ideal $\point$ of $y\in \Y_{m-1}$. Therefore
$$\deg(N)=\sum_{y\in \Y_{m-1}} \length_{B_\point}(N_\point).$$
As $N_\point$ is an Artinian $B_\point$-module and $\dim_k(B/\point)_s=\dim_k (k[T_0])_s=1$ for any $s\geq 0$, one has
\begin{align*}
\length_{B_\point}(N_\point)&=\dim_k(N\otimes_B B_\point) \\
&=\sum_{s}\dim_k(N\otimes_B B_\point)_s\\
&=\sum_{s}\dim_k(H_\m^m(\RI)\otimes_B B_\point)_{(-m,s)} \\
&=\sum_{s}\dim_k(H_\m^m(\RI)\otimes_B B_\point)_{(-m,s)}.\dim_k(B/\point)_s\\
&=\sum_{s}\dim_k H_\m^m(\RI\otimes_B B_\point\otimes_B B/\point)_{(-m,s)}\\
&=\sum_{s}\dim_k H_\m^m(\RI\otimes_B k(\point))_{(-m,s)}\\
&\stackrel{\eqref{equation3.1}}{=}\dim_kH_\m^m(R/(f))_{-m}\\
&\stackrel{\eqref{equation3.3}}{=}\dim_k(R/(f))_{d_f-1}\\
&=\dim_kR_{d_f-1}\qquad \text{since}\quad \deg(f)=d_f\\
&=\binom{d_f+m-1}{m}.
\end{align*}
It follows that 
$$\deg(N)=\sum_{y\in \Y_{m-1}}\binom{\deg(h_y)+m-1}{m}.$$
\end{proof}
 
 \section{Parameterization $\phi: \PP_k^2  \dashrightarrow \PP_k^3$ of surfaces}
 In this section, we consider  a parameterization $\phi: \PP_k^2  \dashrightarrow \PP_k^3$ of surface defined by four homogeneous polynomials $f_0,\ldots,f_3\in R=k[X_0,X_1,X_2]$ of the same degree $d$ such that $\gcd(f_0,\ldots,f_3)=1$. Denote the homogeneous maximal ideal of $R$ by $\m=(X_0,X_1,X_2)$.

 From now on we assume that $\Bc$ is locally a complete intersection. Under this hypothesis, the module $\K$ in the exact sequence 
  $$0\longrightarrow \K \longrightarrow \SI \longrightarrow \RI \longrightarrow 0$$
 is supported in $\m S$. Hence, $H_\m^i(\K)=0$ for any $ i\geq 1.$ The above exact sequence deduces that
 $$H_\m^i(\SI)\simeq H_\m^i(\RI),\, \forall i\geq 1.$$
Let $B=k[T_0,\ldots,T_3]$ be the homogeneous coordinate ring of $\PP_k^3.$ It follows from Theorem~\ref{Theorem3.2} that the finitely generated $B$-module
$$N:=\oplus_{s\geq 0}H_\m^2(I^s)_{sd-2} = H_\m^2(\RI)_{(-2,\ast)} \simeq H_\m^2(\SI)_{(-2,\ast)}$$
satisfying  $\dim(N)=1$ and
$$\Supp_B(N)=\Y_1 =\{ y\in \PP^3_k\mid \dim \pi^{-1}(y)=1  \}.$$
Furthermore,
$$\sum_{y\in \Y_1}\binom{\deg(h_y)+1}{2} =\deg(N) =\dim_k H_\m^2(I^s)_{sd-2}$$
for $s\geq \reg(N)+1,$ where $\reg(N)$ is the Castelnuove-Mumford regularity of $N$. Thus, it is useful to establish the bounds for $\deg(N)$ and $\reg(N)$.

Let $K_\bullet:= K_\bullet(\ff;R)$ and $Z_\bullet:=Z_\bullet(\ff;R)$  be the Koszul complex and the module of cycles associated to the sequence $\ff:=f_0,\ldots,f_3$ with coefficients in $R$, respectively. Since the ideal $I=(\ff)$ is homogeneous, these modules inherit a natural structure of graded $R$-modules. Let $\mathcal{Z}_\bullet$ be the approximation complex associated to $I$. The approximation complexes  were introduced by Herzog, Simis and Vasconcelos in \cite{HSV82} to study the Rees and symmetric algebras of ideals. By definition $$\mathcal{Z}_q=Z_q[qd]\otimes_R R[T_0,\ldots,T_3](-q)$$
for all $q=0,\ldots, 3$ with $\deg(X_i)=(1,0)$ and $\deg(T_i)=(0,1)$. This complex is of the form
\begin{align*}
\xymatrix{
	(\mathcal{Z}_\bullet):& 0\ar[r] &\mathcal{Z}_3\ar[r] ^{v_3}&\mathcal{Z}_2\ar[r]^{v_2}& \mathcal{Z}_1\ar[r]^{v_1\qquad\qquad }&\mathcal{Z}_0=R[T_0,\ldots,T_3]\ar[r]&0 }
\end{align*}
where $v_1(a_0,a_1,a_2,a_3)=a_0T_0+\cdots+a_3T_3$. As $\Bc$ is locally a complete intersection,  the complex  $(\mathcal{Z}_\bullet)$ is acyclic and is a resolution of $H_0(\mathcal{Z}_\bullet)\simeq \SI$, see \cite[Theorem 4]{Buse-Chardin05}.

\begin{pro} \label{Theorem4.1} 
Assume $\Bc$ is locally a complete intersection.  Then $N$ admits a finite representation of free $B$-modules
\begin{align*}
B(-2)^m\longrightarrow B(-1)^n\longrightarrow N\longrightarrow 0,
\end{align*}
where $ n=\dim_k H_\m^1(R/I)_{d-2}$ and $m=\dim_k H_\m^3(Z_2)_{2d-2}.$
\end{pro}
\begin{proof} 
We  consider the two spectral sequences associated to the double complex $C_\m^\bullet(\ZC_\bullet),$ where $C_\m^\bullet(M)$ denotes the \v{C}ech complex on $M$ relatively to the ideal $\m$. Since $(\ZC_\bullet)$ is acyclic, one of them abuts at step two with:
\begin{equation*} 
_\infty\!^h\textbf{E}_{q}^p= \,_2^h\textbf{E}_{q}^p= \begin{cases}  H_\m^p(\SI)\quad \text{for} \quad & q=0\\ 0 \qquad\qquad \text{for}  \quad & q\neq 0.\end{cases}
\end{equation*}
The other one gives at step one:
$$^{v}_1 \textbf{E}_{q}^p=H_\m^p (\ZC_q)=H_\m^p (Z_q)[qd]\otimes_R R[T_0,\ldots,T_3](-q)=H_\m^p (Z_q)[qd]\otimes_k B(-q).$$
By \cite[Lemma 1]{Buse-Chardin05}, $H_\m^p(Z_q)=0$ for $p=0,1$ and by definition, $Z_3\simeq R[-4d]$ and $Z_0=R$. Therefore, the first page of the vertical spectral sequence has only two nonzero lines
{\footnotesize \begin{align*}
\xymatrix@R=8pt{ 0\ar[r]& H_\m^2(Z_2)[2d]\otimes_k B(-2)\ar[r]& H_\m^2(Z_1)[d]\otimes_k B(-1)\ar[r]&0\\
	H_\m^3(Z_3)[3d]\otimes_k B(-3)\ar@{->}[r]& H_\m^3(Z_2)[2d]\otimes_k B(-2)\ar@{->}[r]& H_\m^3(Z_1)[d]\otimes_k B(-1)\ar@{->}[r]& H_\m^3(Z_0)\otimes_kB.	}
\end{align*}}
In bi-degree $(-2,\ast)$, we have $ H_\m^3(Z_0)_{-2}\otimes_k B=H_\m^3(R)_{-2}\otimes_k B=0.$ Therefore, we obtain the complex $(C_\bullet)$ of free $B$-modules
\begin{align*}
\xymatrix@R=8pt{0\ar[r]& B(-3)^l \ar[r]\ar@{=}[d]& B(-2)^m\ar[r] \ar@{=}[d]& B(-1)^n\ar@{=}[d]\ar[r]& 0.\\ & C_3& C_2& C_1&} 
\end{align*}
Notice that $n=\dim_k H^3_\m (Z_1)_{d-2} = \dim_k H^2_\m (I)_{d-2}= \dim_k H^1_\m (R/I)_{d-2}.$

It remains to show that $H_1(C_\bullet)=N.$  It is easy to see that
$$_\infty\! ^{v} \textbf{E}_{q}^p =\ ^{v}_2 \textbf{E}_{q}^p \quad \text{unless} \quad p=q=3\quad \text{or}\quad p=2, q=1.$$
Therefore,
$$\bigoplus_{p-q=2}\,  _\infty\!^{v} \textbf{E}_{q}^p =\ ^{v}_2 \textbf{E}_{1}^3=H_\m^2(\SI)=\bigoplus_{p-q=2}\ _\infty\!^{h} \textbf{E}_{q}^p,$$ 
in other words,
\begin{align*}
H_1(C_\bullet)= H_\m^2(\SI)_{(-2,\ast)} = H_\m^2(\RI)_{(-2,\ast)} =N.
\end{align*} 
\end{proof}
We now establish a bound for the Castelnuovo-Mumford regularity and the degree of $B$-module $N$ in terms of $n=\dim_k H^1_\m (R/I)_{d-2}$ as follows.
\begin{Cor}\label{Cor4.3}
Suppose $\Bc$ is locally a complete intersection.  Then 
$$\reg(N)\leq n\quad \text{and}\quad \deg(N)\leq \binom{n+2}{3}.$$
\end{Cor}
\begin{proof}
As $\dim B = 4$, hence  $\codim(N)=3$ and by Proposition~\ref{Theorem4.1}, $N$ admits a finite representation
\begin{align*}
B(-2)^m\longrightarrow B(-1)^n\longrightarrow N\longrightarrow 0.
\end{align*}
The corollary follows from \cite[Corollaries 2.4 and 3.4]{CFN08}.
\end{proof}
Theorem~\ref{Theorem1.3} shows  that if $\indeg(I^\sat)<d$, then
$$\sum_{y\in \Y_1}\deg(h_y)< d.$$
Hence, the delicate case is when the ideal $I$ satisfies $\indeg (I^\sat)=\indeg(I)= d.$ In this case, the first author in \cite{QHTran17} established an upper bound  for $n=\dim_k H_\m^1(R/I)_{d-2}$ in terms of $d$ as follows.
\begin{pro}\label{Lemma4.2}
Assume $\Bc$ is locally a complete intersection and $\indeg(I^\sat)=d$. Then
	$$\frac{1}{2}d(d+1)\leq\deg(\Bc)\leq d^2-2d+3$$
and $$d\leq n= \deg(\Bc)-\frac{1}{2}d(d-1)\leq \frac{d(d-3)}{2}+3.$$
\end{pro}
 
\section*{Acknowledgments}
 All authors were partially supported by Hue University's Project under grant \#DHH2019-03-114. The authors are grateful to the referee for very carefully reading the manuscript.

\bibliographystyle{vancouver}
\bibliography{bibliothese} 
\end{document}